\documentclass{amsart}

\usepackage{amsmath, amssymb,epic,graphicx,mathrsfs,enumerate}
\usepackage[all]{xy}

\usepackage{latexsym}
\usepackage{longtable}
\usepackage{epsfig}
\usepackage{hhline}

\newtheorem{proposition}{Proposition}[section]

\newtheorem{conjecture}[proposition]{Conjecture}
\newtheorem{lemma}[proposition]{Lemma}
\newtheorem{corollary}[proposition]{Corollary}
\newtheorem{theorem}[proposition]{Theorem}

\begin{document}

\title[A generalization of the diameter bound]{A generalization of the diameter bound of Liebeck and Shalev for finite simple groups}

\author[Attila Mar\'oti]{Attila Mar\'oti}
\address{Alfr\'ed R\'enyi Institute of Mathematics, Hungarian Academy of Sciences, Re\'altanoda utca 13-15, H-1053, Budapest, Hungary}
\email{maroti.attila@renyi.mta.hu}

\author[L\'aszl\'o Pyber]{L\'aszl\'o Pyber}
\address{Alfr\'ed R\'enyi Institute of Mathematics, Hungarian Academy of Sciences, Re\'altanoda utca 13-15, H-1053, Budapest, Hungary}
\email{pyber.laszlo@renyi.mta.hu}

\thanks{The project leading to this application has
received funding from the European Research Council (ERC) under the
European Union's Horizon 2020 research and innovation programme
(grant agreement No. 741420). Both authors were 
partly supported by the National Research, Development and Innovation Office 
(NKFIH) Grant No.~K115799. The first author was also supported by the National Research, Development and Innovation Office 
(NKFIH) Grant No.~K132951.}

\begin{abstract}
Let $G$ be a non-abelian finite simple group. A famous result of Liebeck and Shalev is that there is an absolute constant $c$ such that whenever $S$ is a non-trivial normal subset in $G$ then $S^{k} = G$ for any integer $k$ at least $c \cdot (\log|G|/\log|S|)$. This result is generalized by showing that there exists an absolute constant $c$ such that whenever $S_{1}, \ldots , S_{k}$ are normal subsets in $G$ with $\prod_{i=1}^{k} |S_{i}| \geq {|G|}^{c}$ then $S_{1} \cdots S_{k} = G$.         
\end{abstract}
\maketitle

\begin{center}
{\it To the memory of Jan Saxl.}
\end{center}

\bigskip

\section{Introduction}

A normal subset of a finite group $H$ is defined to be any union of conjugacy classes of $H$. A normal subset is called trivial if it is equal to the identity. A well-known theorem of Liebeck and Shalev \cite[Theorem 1.1]{LS} is that there is a constant $c$ such that whenever $S$ is a non-trivial normal subset in a non-abelian finite simple group $G$ then $S^{k} = G$ for any integer $k$ at least $c \cdot (\log|G|/\log|S|)$.

Gill, Pyber and Szab\'o propose the following conjecture \cite[Conjecture 2]{GPSz}.

\begin{conjecture}
\label{c1}
There exists a constant $c$ such that if $S_1, \ldots , S_k$ are normal subsets of a non-abelian finite simple group $G$ satisfying $\prod_{i=1}^{k} |S_{i}| \geq {|G|}^{c}$, then $S_{1} \cdots S_{k} = G$. 
\end{conjecture}
 
This is a weaker form of another conjecture of Gill, Pyber, Szab\'o \cite[Conjecture 1]{GPSz} which is known to hold \cite[Theorem 2]{GPSz} for finite simple groups of Lie type of bounded rank. 

Rodgers \cite[Corollary 2.4]{R} shows that if $C_{1}, \ldots , C_{k}$ are conjugacy classes of the symmetric group $\mathrm{Sym}(n)$ of degree $n \geq 5$ and $\prod_{i=1}^{k} |C_{i}| > n^{6(n-2)}$, then $C_{1} \cdots C_{k}$ is equal to the alternating group $\mathrm{Alt}(n)$ of degree $n$ or to $\mathrm{Sym}(n) \setminus \mathrm{Alt}(n)$. Let $H$ be a non-solvable special linear group $\mathrm{SL}(n,q)$ or a projective special linear group $\mathrm{PSL}(n,q)$ with $n \geq 2$ and $q$ a prime power. Rodgers and Saxl \cite{RS} prove that if $C_{1}, \ldots , C_{k}$ are conjugacy classes of $H$ with the property that $\prod_{i=1}^{k} |C_{i}| \geq {|H|}^{12}$ then $C_{1} \cdots C_{k} = H$. 

This in particular implies that the constant $c$ in the Liebeck-Shalev theorem \cite[Theorem 1.1]{LS} may be taken to be $12$
in the case of $\mathrm{PSL}(n,q)$. It would be interesting to show that the theorem holds with a similar small constant
$c$ for all non-abelian finite simple groups.

The aim of this paper is to prove Conjecture \ref{c1}.  

\begin{theorem}
\label{main}
Conjecture \ref{c1} is true. 
\end{theorem}

The argument relies on the Liebeck-Shalev Theorem \cite[Theorem 1.1]{LS} via \cite[Proposition 5.2]{GPSSz}, on deep character theoretic results of Guralnick, Larsen and Tiep \cite[Theorem 1.3]{GLT} and of Liebeck and Shalev \cite[Theorem 1.1]{LS2}, and in the case of alternating groups on results of Rodgers \cite{R}. We actually give a new, different proof of the result of Rodgers and Saxl \cite{RS} with a weaker, non-explicit constant in the exponent. 

A by-product of the proof of Theorem \ref{main} is the following result.  

\begin{theorem}
\label{uj}
Let $G$ be a non-abelian finite simple group. There exists a constant $\delta$ with $0 < \delta < 1$ such that if $S_{1}, \ldots , S_{8}$ are normal subsets in $G$ each of size at least ${|G|}^{\delta}$, then $S_{1} \cdots S_{8} = G$ in case $G$ is an alternating group and $S_{1}S_{2}S_{3} = G$ in case $G$ is different from an alternating group.
\end{theorem}

For $G$ a classical simple group, Theorem \ref{uj} was obtained independently using different techniques by Larsen, Shalev, Tiep in \cite[Theorem 7.4]{LST}.

Note that if $\delta$ is chosen appropriately, Theorem \ref{uj} is trivially true for groups of orders less than any given universal constant and it is easy to establish for groups of Lie type of bounded rank, in particular for exceptional groups. 

Theorem \ref{uj} implies various results in \cite{Shalev}. In particular, it improves the following key result \cite[Corollary 2.5]{Shalev}. For every $\epsilon > 0$ there is a number $r(\epsilon)$ such that whenever $G$ is a finite simple group of Lie type of Lie rank $r \geq r(\epsilon)$ defined over the field with $q$ elements, $C_{1}$, $C_{2}$, $C_{3}$ are conjugacy classes of $G$, $x_{1} \in C_{1}$, $x_{2} \in C_{2}$, $x_{3} \in C_{3}$, and $|C_{G}(x_{1})| |C_{G}(x_{2})| |C_{G}(x_{3})| \leq q^{(4-\epsilon)r}$, then $C_{1}C_{2}C_{3} = G$.

Theorem \ref{uj} could be considered as an alternative approach to a conjecture of Thompson. Let $G$ be a non-abelian finite simple group. Thompson's Conjecture states that $G$ has a conjugacy class $C$ such that $C^{2} = G$. This is established for alternating groups \cite{C} and for finite simple groups of Lie type \cite{EllersGordeev} defined over fields of size larger than $8$. 

There are other results in the literature which may be considered as approximations of Thompson's conjecture. For every sufficiently large $G$ there is a conjugacy class $C$ such that $C^{3} = G$ by \cite[Corollary 2.3]{Shalev}. It is also known \cite[Theorem 1.4]{GuralnickMalle} that for every $G$ there are conjugacy classes $C$ and $D$ of $G$ such that $CD \cup \{ 1 \} = G$ (this is an extension of results in \cite{MalleSaxlWeigel} and also of \cite[Theorem 1.1.4]{LarsenShalevTiep}).   

It would be interesting to find the smallest integer $r$ with $3 \leq r \leq 8$ such that in Theorem \ref{uj} we get $S_{1} \cdots S_{r} = G$. (Note that if $S_{1}$ and $S_{2}$ are conjugacy classes (of size at least $|G|^{\delta}$) with $S_{1} \not= S_{2}^{-1}$ then $S_{1}S_{2} \not= G$.) It would also be interesting to determine the smallest value of $\delta$ for which Theorem \ref{uj} holds.

\section{Small normal sets}

In this section it is shown that in order to prove Theorem \ref{main} we may assume that each normal subset $S_i$ of $G$ is large. 

The starting point is \cite[Proposition 5.2]{GPSSz}.

\begin{lemma}[Gill, Pyber, Short, Szab\'o]
\label{GPSSz}
For every $\delta$ with $0 < \delta < 1$ there exists $\epsilon > 0$ such that for any non-abelian finite simple group $G$ and subsets $A$ and $B$ of $G$ with $B$ normal in $G$ and $|A| \leq {|G|}^{\delta}$ we have $$|AB| \geq |A|{|B|}^{\epsilon}.$$ 
\end{lemma}

Note again that the proof of Lemma \ref{GPSSz} depends on \cite[Theorem 1.1]{LS}. In the course of the proof of Theorem \ref{main} this is the only place where \cite[Theorem 1.1]{LS} is used. 

Fix $\delta$ and let $\epsilon > 0$ be a constant whose existence is assured by Lemma \ref{GPSSz}.

\begin{lemma}
\label{l1}
Let $G$ be a non-abelian finite simple group and $t$ an integer at least $2$. If $A_{1}, \ldots , A_{t}$ are normal subsets in $G$ with $|A_{1} \cdots A_{t-1}| \leq {|G|}^{\delta}$, then $$|A_{1} \cdots A_{t}| \geq (|A_1| \cdots |A_{t}|)^{\epsilon}.$$
\end{lemma}

\begin{proof}
This follows by applying Lemma \ref{GPSSz} $t-1$ times.  
\end{proof}

Lemma \ref{l1} has an immediate consequence. 

\begin{corollary}
\label{newcor}
Let $G$ be a non-abelian finite simple group and let $A_{1}, \ldots , A_{j}$ be normal subsets in $G$. Let $t$ be the least integer such that 
$|A_{1}| \cdots |A_{t}| > {|G|}^{1/\epsilon}$. Then $|A_{1} \cdots A_{t}| \geq |G|^{\delta}$.
\end{corollary}

\begin{proof}
If $|A_{1} \cdots A_{t}| < {|G|}^{\delta}$, then $|A_{1} \cdots A_{t}| \geq (|A_1| \cdots |A_{t}|)^{\epsilon} > |G|$ by Lemma \ref{l1}. A contradiction. 
\end{proof}

Corollary \ref{newcor} is applied in the following. 

\begin{lemma}
\label{ez}
In proving Theorem \ref{main} (i.e. Conjecture \ref{c1}) we may assume that each normal subset $S_i$ has size at least ${|G|}^{\delta}$.
\end{lemma} 

\begin{proof}
Assume that there exists a constant $a$ such that whenever $j$ is an integer at least $a$ and $A_{1}, \ldots , A_{j}$ are normal subsets of $G$ satisfying $|A_{i}| \geq {|G|}^{\delta}$ for every $i$ with $1 \leq i \leq j$, then $A_{1} \cdots A_{j} = G$. 

Let $S_{1}, \ldots , S_{k}$ be arbitrary normal subsets in $G$ such that $\prod_{i=1}^{k} |S_{i}| \geq {|G|}^{a (1 + 1/\epsilon)}$. We claim that $S_{1} \cdots S_{k} = G$. 

We choose the smallest number $t_{1}$ such that $|S_{1}| \cdots |S_{t_1}| > {|G|}^{1/\epsilon}$. It is clear that $|S_{1}| \cdots |S_{t_{1}}| \leq {|G|}^{1 + 1/\epsilon}$ and by Corollary \ref{newcor} we have $|S_{1} \cdots S_{t_{1}}| \geq {|G|}^{\delta}$.

Similarly, we choose numbers $0 = t_{0}$, $t_{1}, \ldots , t_{j}$ with $1 \leq t_{1} < t_{2} < \ldots < t_{j} \leq k$ such that for every $i$ with $0 \leq i \leq j-1$ we have ${|G|}^{1/\epsilon} < |S_{t_{i}+1}| \cdots |S_{t_{i+1}}| \leq {|G|}^{1 + 1/\epsilon}$ and hence 
$|A_{i+1}| \geq {|G|}^{\delta}$ where $A_{i+1} = S_{t_{i}+1} \cdots S_{t_{i+1}}$.

The condition $\prod_{i=1}^{k} |S_{i}| \geq {|G|}^{a (1 + 1/\epsilon)}$ implies $j \geq a$. The lemma now follows from the first paragraph of the proof. 
\end{proof}


Theorem \ref{main} follows from Lemma \ref{ez} together with Theorem \ref{uj}. The rest of the paper is devoted to the proof of Theorem \ref{uj}.

\section{Groups of bounded rank}

Theorem \ref{uj} is true for every non-abelian simple group $G$ of order at most $2^{1/(1-\delta)}$. To see this observe that under this condition each $S_{i}$ has size larger than $|G|/2$, and apply the following lemma found in \cite[p. 58; 10]{Jacobson}.  

\begin{lemma}
\label{Jacobson}
If $A$ and $B$ are subsets of a finite group $G$ with $|A|+|B| > |G|$, then $AB = G$.
\end{lemma}  

\begin{proof}
Fix $g \in G$. Since $|A| + |gB^{-1}| > |G|$, we have $A \cap gB^{-1} \not= \emptyset$ and therefore $g = a b$ for some $a \in A$ and $b \in B$. 
\end{proof}

It is mentioned in the Introduction that a stronger form of Conjecture \ref{c1} is known to hold for finite simple groups of Lie type of bounded rank. This is proved by Gill, Pyber, Szab\'o \cite[Theorem 2]{GPSz} using the Product Theorem. 


We now prove Theorem \ref{uj} for finite simple groups of bounded rank. Together with Lemma \ref{ez} this yields a shorter and more direct proof of Conjecture \ref{c1} in this case. 

It is observed by Nikolov and Pyber in \cite{NikolovPyber} that a result of Gowers \cite{Gowers} implies the following. 

\begin{lemma}  
\label{trick}
Let $G$ be a finite group and let $m$ denote the dimension of the smallest non-trivial complex irreducible representation of $G$. If $A$, $B$, $C$ are subsets of $G$ such that $|A||B||C| \geq |G|^{3}/m$, then $ABC = G$. 
\end{lemma}

Let $G$ be a finite simple group of Lie type of rank $r$. As noted in \cite[Section 2]{GPSz}, using Lemma \ref{trick}, it follows that if $A$, $B$, $C$ are subsets of $G$ each of size larger than ${|G|}^{1- 1/(24 r^{2})}$, then $ABC = G$. 

Choose $\delta$ such that $\delta > 1 - 1/(24 r^{2})$ holds. Theorem \ref{uj} then follows for finite simple groups of Lie type of rank at most $r$ since $S_{1}S_{2}S_{3} = G$ by the previous paragraph.   

\begin{lemma}
\label{nagy}
In proving Theorem \ref{uj} we may assume that $G$ is an alternating group $\mathrm{Alt}(n)$ or $G$ is a classical simple group $\mathrm{Cl}(n,q)$ and in both cases we may assume that $n$ is sufficiently large. 
\end{lemma}

From now on assume that $G$ is an alternating or a classical simple group. 

\section{Large conjugacy classes}

The aim of this section is to pass in Theorem \ref{uj} from (large) normal subsets to large conjugacy classes. 

Let $k(H)$ denote the number of conjugacy classes of a finite group $H$. Part (i) of the next lemma is due to Kov\'acs and Robinson \cite[Lemma 1.1]{KR}, while part (ii) follows from a special case of a result of Liebeck and Pyber \cite[Theorem 1.1]{LP}. 

\begin{lemma}
\label{k}
The following hold.
\begin{enumerate}

\item[(i)] $k(\mathrm{Alt}(n)) \leq 2^{n-1}$ for $n \geq 5$.

\item[(ii)] $k(\mathrm{Cl}(n,q)) \leq q^{d n}$ for some constant $d$.

\end{enumerate}
\end{lemma} 

Fix $\alpha$ with $0 < \alpha < \delta < 1$. 

\begin{lemma}
\label{l3}
If $A$ is a normal subset of $G$ with $|A| \geq {|G|}^{\delta}$, then $A$ contains a conjugacy class $C$ of $G$ with $|C| \geq {|G|}^{\alpha}$, at least for $n$ sufficiently large.  
\end{lemma}
 
\begin{proof}
Any normal subset $A$ of $G$ contains a conjugacy class $C$ of $G$ of size at least $|A|/k(G)$. If in addition it is assumed that $|A| \geq {|G|}^{\delta}$, then $|C| \geq {|G|}^{\alpha}$ by Lemma \ref{k}, at least for $n$ sufficiently large. 
\end{proof}

Let $r = 8$ if $G = \mathrm{Alt}(n)$ and $r = 3$ if $G = \mathrm{Cl}(n,q)$. Assume that there is a choice of $\alpha$ such that whenever $C_{1}, \ldots , C_{r}$ are conjugacy classes in $G$ each of size at least ${|G|}^{\alpha}$, then $C_{1} \cdots C_{r} = G$. For sufficiently large $n$, each $S_i$ contains a conjugacy class $C_i$ of $G$ of size at least ${|G|}^{\alpha}$ by Lemma \ref{l3}. Thus 
$S_{1} \cdots S_{r} \supseteq C_{1} \cdots C_{r} = G$. 

The following is proved. 

\begin{lemma}
\label{ezz}
In proving Theorem \ref{uj} we may assume that each of the normal subsets $S_i$ is a conjugacy class of size at least ${|G|}^{\alpha}$.
\end{lemma}

\section{Alternating groups}

In this section we prove Theorem \ref{uj} in the case when $G = \mathrm{Alt}(n)$ and $n \geq 9$ is sufficiently large. 

For each index $i$ with $1 \leq i \leq 8$, the normal set $S_{i}$ is a conjugacy class by Lemma \ref{ezz} with 
$$|S_{i}| \geq {|G|}^{\alpha} = {|\mathrm{Alt}(n)|}^{\alpha} = {\Big(\frac{n!}{2}\Big)}^{\alpha} > \frac{1}{2} {\Big(\frac{n}{2}\Big)}^{\alpha n} > n^{\beta n},$$ for any constant $\beta$ with $\beta < \alpha$, provided that $n$ is sufficiently large. 

A key invariant in \cite{R} is the following. Let $C$ be a conjugacy class of $\mathrm{Alt}(n)$ or of $\mathrm{Sym}(n)$. Define $\delta(C)$ to be $n-t$ where $t$ is the number of orbits of $\langle x \rangle$ (on the underlying set $\{ 1, \ldots , n \}$) for an element $x \in C$. It is implicit in the proof of \cite[Corollary 2.4]{R} that $|C| \leq n^{2 \delta(C)}$. 

The previous two paragraphs imply $\delta(S_{i}) > \beta n / 2$ for every $i$ with $1 \leq i \leq 8$.   

A special case of \cite[Theorem 2.3]{R} is the following. 

\begin{lemma}[Rodgers]
\label{Rodgers2}
Let $C_{1}, \ldots , C_{r}$ be conjugacy classes of $\mathrm{Sym}(n)$ such that every conjugacy class is contained in $\mathrm{Alt}(n)$. If $\sum_{i=1}^{r} \delta(C_{i}) > 3(n-2)$ and $n \geq 5$, then $C_{1} \cdots C_{r} = \mathrm{Alt}(n)$.
\end{lemma}

Let $\ell \geq 0$ be the number of $S_{i}$ which are conjugacy classes not only of $\mathrm{Alt}(n)$ but also of $\mathrm{Sym}(n)$. Without loss of generality, let these classes be $S_{1}, \ldots , S_{\ell}$. 

Each $S_{i}$ with $i$ between $\ell+1$ and $8$ has the property that it is not a conjugacy class in $\mathrm{Sym}(n)$, that is, for every $x \in S_{\ell+1} \cup \ldots \cup S_{8}$ the disjoint cycles in $x$ have pairwise different odd lengths. 

Let $x$ be an arbitrary element in $S_{\ell + 1} \cup \ldots \cup S_{8}$. Let $t$ be the number of orbits of $\langle x \rangle$. We claim that $t \leq \sqrt{n}$. Let $s$ be the largest odd integer such that $${(s+1)}^{2}/4 = 1 + 3 + \ldots + (s-2) + s \leq n.$$  If $t > (s+1)/2$, then $x$ must act on at least ${(s+1)}^{2}/4 + (s + 2) > n$ points. This is a contradiction, so $t \leq (s+1)/2 \leq \sqrt{n}$. 

We claim that for any $i$ and $j$ with $\ell + 1 \leq i < j \leq 8$ and any $n \geq 9$, there is a conjugacy class $C_{i,j}$ of $\mathrm{Sym}(n)$ and also of $\mathrm{Alt}(n)$ such that $C_{i,j} \subseteq S_{i}S_{j}$ and $$\delta(C_{i,j}) \geq n - 2\sqrt{n} - 2.$$ 

Let $x$ and $y$ be elements of $S_i$ and $S_j$ respectively, chosen in such a way that in the disjoint cycle decompositions of $x$ and $y$ the numbers $1$ through $n$ are in increasing order and the cycle lengths are in decreasing order. If $n \geq 9$, then the numbers $1$, $2$, $3$, $4$ all appear in the longest cycles of $x$ and $y$. Let $y'$ be the conjugate permutation $y^{(12)(34)} \in S_{j}$. The permutation $xy' \in S_{i}S_{j}$ fixes $1$ and $3$ and thus, having at least two fixed points, the conjugacy class $C_{i,j}$ of $\mathrm{Alt}(n)$ containing $xy'$ is also a conjugacy class of $\mathrm{Sym}(n)$. Observe that the number of orbits of $\langle xy' \rangle$ is at most the number of integers $i$ with $1 \leq i \leq n$ such that $i(xy') \leq i$. This is at most  
$$|\{ i \ : \ ix \leq i \}| + |\{ i \ : \ ix > i \ \mathrm{and} \ (ix)y' < ix \}| \leq$$
$$\leq |\{ i \ : \ ix \leq i \}| + |\{ i \ : \ iy' \leq i \}| \leq 2 \sqrt{n} + 2$$ 
since both $\langle x \rangle$ and $\langle y \rangle$ have at most $\sqrt{n}$ orbits. It follows that $\delta(C_{i,j}) \geq n - 2\sqrt{n} - 2$.

The above give $$\delta(S_{1}) + \cdots + \delta(S_{\ell}) + \delta(C_{\ell+1,\ell+2}) + \cdots + \delta(C_{7,8})
\geq \frac{\beta \ell}{2} n + \Big[ \frac{8-\ell}{2} \Big] (n - 2\sqrt{n} - 2),$$
where $[ (8-\ell) / 2 ]$ denotes the integer part of $(8-\ell)/2$. 

Now choose $\beta$ larger than $7/8$.

Since $\beta > 7/8$ and $\ell \leq 8$, we have 
$$\frac{\beta \ell}{2} n + \Big[ \frac{8-\ell}{2} \Big] (n - 2\sqrt{n} - 2) > \Big(\frac{7}{2} - \frac{\ell}{16}\Big) \cdot n > 3(n-2),$$ for every sufficiently large $n$. Finally Lemma \ref{Rodgers2} gives $$S_{1} \cdots S_{8} \supseteq S_{1} \cdots S_{\ell} \cdot C_{\ell+1,\ell+2} \cdots C_{7,8} = \mathrm{Alt}(n).$$ 

This finishes the proof of Theorem \ref{uj} in case $G = \mathrm{Alt}(n)$.

\section{Classical simple groups}

In this section the proof of Theorem \ref{uj} is completed. 

It may be assumed by Lemma \ref{nagy} and the previous section that $G$ is a classical simple group $\mathrm{Cl}(n,q)$ with sufficiently large $n$. 

A special case of \cite[Theorem 1.3]{GLT} is the following. 

\begin{lemma}[Guralnick, Larsen, Tiep]
\label{GLT}
There exists a $\mu > 0$ such that whenever $G$ is a classical simple group and $g \in G$ satisfies $|C_{G}(g)| \leq {|G|}^{\mu}$, then $|\chi(g)| \leq \chi(1)^{1/10}$ for every $\chi \in \mathrm{Irr}(G)$. 
\end{lemma}

Let $\zeta^{H}(t) = \sum_{\chi \in \mathrm{Irr}(H)} \chi(1)^{-t}$ for any finite group $H$. A special case of \cite[Theorem 1.1]{LS2} is the following.

\begin{lemma}[Liebeck, Shalev]
\label{zeta}
For any sequence of non-abelian finite simple groups $H \not= \mathrm{PSL}(2,q)$ and any $t > 2/3$, $\zeta^{H}(t) \to 1$ as $|H| \to \infty$.
\end{lemma}

Let $D_{1}$, $D_{2}$, $D_{3}$ be conjugacy classes of $G$ each of size at least ${|G|}^{1-\mu}$ where $\mu$ is as in Lemma \ref{GLT}.

\begin{lemma}
\label{four1}
We have $D_{3} \subseteq D_{1} D_{2}$ for every sufficiently large $n$.  
\end{lemma} 

\begin{proof}
Let $d_{1} \in D_{1}$, $d_{2} \in D_{2}$, and $d_{3} \in D_{3}$. The conjugacy class $D_{3}$ is contained in the normal set $D_{1}D_{2}$ if and only if the non-negative rational number 
$$f(D_{1}, D_{2}, D_{3}):= \sum_{\chi \in \mathrm{Irr}(G)} \frac{\chi(d_{1}) \chi(d_{2}) \chi(d_{3}^{-1})}{\chi(1)} =
1 + \sum_{1 \not= \chi \in \mathrm{Irr}(G)} \frac{\chi(d_{1}) \chi(d_{2}) \chi(d_{3}^{-1})}{\chi(1)}$$
is positive, by \cite[p. 43]{ASH}.

Let $g \in D_{1} \cup D_{2} \cup D_{3}$ be arbitrary. We have $|C_{G}(g)| \leq {|G|}^{\mu}$. Thus $|\chi(g)| \leq \chi(1)^{1/10}$ for every $\chi \in \mathrm{Irr}(G)$ by Lemma \ref{GLT}. 

It follows that $$\Big| \sum_{1 \not= \chi \in \mathrm{Irr}(G)} \frac{\chi(d_{1}) \chi(d_{2}) \chi(d_{3}^{-1})}{\chi(1)} \Big| \leq 
\sum_{1 \not= \chi \in \mathrm{Irr}(G)} \Big| \frac{\chi(d_{1}) \chi(d_{2}) \chi(d_{3}^{-1})}{\chi(1)} \Big| =$$ 
$$= \sum_{1 \not= \chi \in \mathrm{Irr}(G)} \frac{|\chi(d_{1})| |\chi(d_{2})| |\chi(d_{3}^{-1})|}{\chi(1)} \leq 
\sum_{1 \not= \chi \in \mathrm{Irr}(G)} \chi(1)^{- 7/10} = \zeta^{G}(7/10) - 1.$$

We have $\zeta^{G}(7/10) - 1 \to 0$ as $n \to \infty$, by Lemma \ref{zeta}. Thus $f(D_{1}, D_{2}, D_{3}) > 0$, provided that $n$ is sufficiently large. 
\end{proof}

Choose $\gamma > 0$ (and thus $\alpha$) such that $0 < 1 - \mu \leq \gamma < \alpha < 1$ where $\mu > 0$ is a constant whose existence is assured by Lemma \ref{GLT}. 

The following lemma finishes the proof of Theorem \ref{uj}. 

\begin{lemma}
We have $S_{1}S_{2}S_{3} = G$ for every sufficiently large $n$. 
\end{lemma}

\begin{proof}
Let $n$ be sufficiently large.

By Lemma \ref{ezz} we may assume that $S_{1}$, $S_{2}$ and $S_{3}$ are conjugacy classes of size at least $|G|^{\alpha}$. The normal set $S_{1}S_{2}$ contains every conjugacy class of $G$ of size at least ${|G|}^{\gamma}$ by Lemma \ref{four1}. Thus $$|S_{1}S_{2}| \geq |G| - k(G) {|G|}^{\gamma} \geq |G| - q^{dn} {|G|}^{\gamma}$$ by Lemma \ref{k}. It follows that 
$$|S_{1}S_{2}| + |S_{3}| \geq |G| - q^{dn} {|G|}^{\gamma} + {|G|}^{\alpha} > |G|$$ if $n$ (hence $|G|$) is sufficiently large. Now, Lemma \ref{Jacobson} gives $(S_{1}S_{2}) S_{3} = G$. 
\end{proof}

This also completes the proof of Theorem \ref{main}.


\begin{thebibliography}{30}

\bibitem{ASH} Z. Arad, M. Herzog, J. Stavi, Powers and products of conjugacy classes in groups. Products of conjugacy classes in groups, 6--51,
Lecture Notes in Math., 1112, Springer, Berlin, 1985. 

\bibitem{EllersGordeev} E. W. Ellers, N. Gordeev, On the conjectures of J. Thompson and O. Ore. \emph{Trans. Amer. Math. Soc.} \textbf{350} (1998), no. 9, 3657--3671. 

\bibitem{GPSSz} N. Gill, L. Pyber, I. Short, E. Szab\'o, On the product decomposition conjecture for finite simple groups.
\emph{Groups Geom. Dyn.} \textbf{7} (2013), no. 4, 867--882. 

\bibitem{GPSz} N. Gill, L. Pyber, E. Szab\'o, A generalization of a theorem of Rodgers and Saxl for simple groups of bounded rank. ArXiv:1901.09255.

\bibitem{Gowers} T. W. Gowers, Quasirandom groups. \emph{Combin. Probab. Comput.} \textbf{17} (2008), no. 3, 363--387. 

\bibitem{GLT} R. M. Guralnick, M. Larsen, P. H. Tiep, Character levels and character bounds. II. ArXiv:1904.08070v2.

\bibitem{GuralnickMalle} R. M. Guralnick, G. Malle, Products of conjugacy classes and fixed point spaces. \emph{J. Amer. Math. Soc.} \textbf{25} (2012), no. 1, 77--121.

\bibitem{Jacobson} N. Jacobson, Basic algebra. I. Second edition. W. H. Freeman and Company, New York, 1985.

\bibitem{KR} L. G. Kov\'acs, G. R. Robinson, On the number of conjugacy classes of a finite group. \emph{J. Algebra} \textbf{160} (1993), no. 2, 441--460.

\bibitem{LarsenShalevTiep} M. Larsen, A. Shalev, P. H. Tiep, The Waring problem for finite simple groups. \emph{Ann. of Math.} (2) \textbf{174} (2011), no. 3, 1885--1950.

\bibitem{LST} M. Larsen, A. Shalev, P. H. Tiep, Products of normal subsets and derangements. ArXiv:2003.12882.

\bibitem{LP} M. W. Liebeck, L. Pyber, Upper bounds for the number of conjugacy classes of a finite group. \emph{J. Algebra} \textbf{198} (1997), no. 2, 538--562. 

\bibitem{LS} M. W. Liebeck, A. Shalev, Diameters of finite simple groups: sharp bounds and applications. \emph{Ann. of Math.} (2) \textbf{154} (2001), no. 2, 383--406.

\bibitem{LS2} M. W. Liebeck, A. Shalev, Fuchsian groups, finite simple groups and representation varieties. \emph{Invent. Math.} \textbf{159} (2005), no. 2, 317--367. 

\bibitem{MalleSaxlWeigel} G. Malle, J. Saxl, T. Weigel, Generation of classical groups. \emph{Geom. Dedicata} \textbf{49} (1994), no. 1, 85--116. 

\bibitem{NikolovPyber} N. Nikolov, L. Pyber, Product decompositions of quasirandom groups and a Jordan type theorem. \emph{J. Eur. Math. Soc. (JEMS)} \textbf{13} (2011), no. 4, 1063--1077.

\bibitem{R} D. M. Rodgers, Generating and covering the alternating or symmetric group. \emph{Comm. Algebra} \textbf{30} (2002), no. 1, 425--435.

\bibitem{RS} D. M. Rodgers, J. Saxl, Products of conjugacy classes in the special linear groups. \emph{Comm. Algebra} \textbf{31} (2003), no. 9, 4623--4638. 

\bibitem{Shalev} A. Shalev, Word maps, conjugacy classes, and a noncommutative Waring-type theorem. \emph{Ann. of Math.} (2) \textbf{170} (2009), no. 3, 1383--1416.

\bibitem{C} Xu Cheng-hao, The commutators of the alternating group. \emph{Sci. Sinica} \textbf{14} (1965), 339--342. 

\end{thebibliography}
\end{document}